\documentclass{amsart}
\usepackage{amscd,amsthm,amssymb,amsfonts,amsmath,euscript}

\theoremstyle{plain}
\newtheorem{thm}{Theorem}[section]
\newtheorem{lemma}[thm]{Lemma}
\newtheorem{prop}[thm]{Proposition}
\newtheorem{cor}[thm]{Corollary}

\theoremstyle{definition}

\theoremstyle{remark}
\newtheorem{remark}[thm]{Remark}

\newcommand{\nc}{\newcommand}

\def\makeop#1{\expandafter\def\csname#1\endcsname
  {\mathop{\rm #1}\nolimits}\ignorespaces}
\makeop{Hom}   \makeop{End}   \makeop{Aut}   \makeop{Isom}  \makeop{Pic}
\makeop{Gal}   \makeop{ord}   \makeop{Char}  \makeop{Div}   \makeop{Lie}
\makeop{PGL}   \makeop{Corr}  \makeop{PSL}   \makeop{sgn}   \makeop{Spf}
\makeop{Spec}  \makeop{Tr}    \makeop{Nr}    \makeop{Fr}    \makeop{disc}
\makeop{Proj}  \makeop{supp}  \makeop{ker}   \makeop{im}    \makeop{dom}
\makeop{coker} \makeop{Stab}  \makeop{SO}    \makeop{SL}    \makeop{SL}
\makeop{Cl}    \makeop{cond}  \makeop{Br}    \makeop{inv}   \makeop{rank}
\makeop{id}    \makeop{Fil}   \makeop{Frac}  \makeop{GL}    \makeop{SU}
\makeop{Trd}   \makeop{Sp}    \makeop{Tr}    \makeop{Trd}   \makeop{diag}
\makeop{Res}   \makeop{ind}   \makeop{depth} \makeop{Tr}    \makeop{st}
\makeop{Ad}    \makeop{Int}   \makeop{tr}    \makeop{Sym}   \makeop{can}
\makeop{length}\makeop{SO}    \makeop{torsion} \makeop{GSp} \makeop{Ker}
\def\makebb#1{\expandafter\def
  \csname bb#1\endcsname{{\mathbb{#1}}}\ignorespaces}
\def\makebf#1{\expandafter\def\csname bf#1\endcsname{{\bf
      #1}}\ignorespaces}
\def\makegr#1{\expandafter\def
  \csname gr#1\endcsname{{\mathfrak{#1}}}\ignorespaces}
\def\makescr#1{\expandafter\def
  \csname scr#1\endcsname{{\EuScript{#1}}}\ignorespaces}
\def\makecal#1{\expandafter\def\csname cal#1\endcsname{{\mathcal
      #1}}\ignorespaces}

\def\doLetters#1{#1A #1B #1C #1D #1E #1F #1G #1H #1I #1J #1K #1L #1M
                 #1N #1O #1P #1Q #1R #1S #1T #1U #1V #1W #1X #1Y #1Z}
\def\doletters#1{#1a #1b #1c #1d #1e #1f #1g #1h #1i #1j #1k #1l #1m
                 #1n #1o #1p #1q #1r #1s #1t #1u #1v #1w #1x #1y #1z}
\doLetters\makebb   \doLetters\makecal  \doLetters\makebf
\doLetters\makescr
\doletters\makebf   \doLetters\makegr   \doletters\makegr
     \def\qed{\qedmark\medbreak}%
\def\qedmark{{\enspace\vrule height 6pt width 5pt depth 1.5pt}}%
    
   \def\Ga{{{\bbG}_{\rm a}}}

\normalsize

\makeop{Bl}

\def\Fpbar{\overline{\bbF}_p}

\def\Qpbar{\overline{{\bbQ}_p}}

\def\Qbar{\overline{\bbQ}}

\newcommand{\Z}{\mathbb Z}
\newcommand{\Q}{\mathbb Q}

\newcommand{\C}{\mathbb C}
\newcommand{\A}{\mathbb A}    
\newcommand{\F}{\mathbb F}

\newcommand{\npr}{\noindent }

\newcommand{\<}{\langle}   
\renewcommand{\>}{\rangle} 

\nc{\embed}{\hookrightarrow}

\newcommand{\ch}{characteristic }
\newcommand{\ac}{algebraically closed }
\newcommand{\dieu}{Dieudonn\'{e} }

\nc{\ol}{\overline}
\nc{\wt}{\widetilde}
\nc{\opp}{\mathrm{opp}}
\def\ul{\underline}

\makeop{Ram}
\makeop{Rep}

\begin{document}
\renewcommand{\thefootnote}{\fnsymbol{footnote}}
\setcounter{footnote}{-1}
\numberwithin{equation}{section}


\title{Mass formula for supersingular abelian surfaces}
\author{Chia-Fu Yu and Jeng-Daw Yu}
\address{
(C.-F.~Yu) Institute of Mathematics \\
Academia Sinica \\
128 Academia Rd.~Sec.~2, Nankang\\
Taipei, Taiwan \\ and NCTS (Taipei Office)}
\email{chiafu@math.sinica.edu.tw}
\address{
(J.-D.~Yu) Department of Mathematics \\
National Taiwan University \\
Taipei, Taiwan}
\email{jdyu@math.ntu.edu.tw}

\date{\today.  
}
\begin{abstract}
We show a mass formula for arbitrary supersingular abelian surfaces in
\ch $p$. 
\end{abstract}

\maketitle


\section{Introduction}
\label{sec:01}
\def\Mass{{\rm Mass}}

In \cite{chai:ho} Chai studied prime-to-$p$ Hecke correspondences on
Siegel moduli spaces in \ch $p$ and proved a deep geometric result
about ordinary $\ell$-adic Hecke orbits for any prime $\ell\neq
p$. Recently Chai and Oort gave a complete answer to what this $\ell$-adic
Hecke orbit can be; see \cite{chai:sketch}.
In this paper we study the arithmetic aspect of supersingular
$\ell$-adic Hecke orbits in the Siegel moduli spaces, the extreme
situation opposite to the ordinary case.
In the case of genus $g=2$, we give a
complete answer to the size of supersingular Hecke orbits.

Let $p$ be a rational prime number and $g\ge 1$ be a positive
integer. Let $N\ge 3$ be a prime-to-$p$
positive integer. Choose a primitive
$N$th root of unity $\zeta_N\in \Qbar\subset \C$ and fix an
embedding $\Qbar \hookrightarrow \Qpbar$. Let $\calA_{g,1,N}$
denote the moduli space over $\Fpbar$ of $g$-dimensional principally
polarized abelian varieties with a symplectic level-$N$ structure
with respect to $\zeta_N$. Let $k$ be an \ac field of \ch $p$.
For each point $x=\ul A_0=(A_0,\lambda_0,\eta_0)$ in
$\calA_{g,1,N}(k)$ and a prime number $\ell\neq p$, 
the $\ell$-adic Hecke orbit $\calH_\ell(x)$ is defined to be the countable
subset of $\calA_{g,1,N}(k)$ that consists of points $\ul A$ such that
there is an $\ell$-quasi-isogeny from $A$ to $ A_0$ that
preserves the polarizations (see \S \ref{sec:02} for definitions). 
%
%
%
It is proved in Chai \cite[Proposition
1]{chai:ho} that the $\ell$-adic Hecke orbit $\calH_\ell(x)$ is finite 
if and only if $x$ is supersingular. Recall that an 
abelian variety $A$ over $k$ is
{\it supersingular} if it is isogenous to a product of supersingular elliptic
curves; $A$ is {\it superspecial} if it is isomorphic to a product of
supersingular elliptic curves. A natural question is whether it is
possible to calculate the size of a supersingular Hecke orbit. The
answer is affirmative, provided that we know its underlying $p$-divisible
group structure explicitly, through the calculation of geometric mass formulas 
(see Section~\ref{sec:02}). This is the task of this paper
where we examine the $p$-divisible group structure of
some non-superspecial abelian varieties. 

Let $x=(A_0,\lambda_0)$ be a $g$-dimensional supersingular principally
polarized abelian varieties over $k$. Let $\Lambda_x$ denote the set
of isomorphism classes of $g$-dimensional supersingular principally
polarized abelian varieties $(A,\lambda)$ over $k$ such that
there exists an isomorphism
$(A,\lambda)[p^\infty]\simeq (A_0,\lambda_0)[p^\infty]$
of the attached quasi-polarized $p$-divisible groups; it is a
finite set (see \cite[Theorem 2.1 and Proposition 2.2]{yu:mass_hb}).
Define the mass $\Mass(\Lambda_x)$ of $\Lambda_x$ as
\begin{equation}
  \label{eq:11}
  \Mass(\Lambda_x):=\sum_{(A,\lambda)\in \Lambda_x}
  \frac{1}{|\Aut(A,\lambda)|}.
\end{equation}
The main result of this paper is computing the geometric mass
$\Mass(\Lambda_x)$ for arbitrary $x$ when $g=2$.

Let $\Lambda_{2,p}^*$ be the set of isomorphism classes of
polarized superspecial abelian surfaces $(A,\lambda)$
with polarization degree
$\deg \lambda=p^2$ over $\Fpbar$ such that $\ker \lambda\simeq
\alpha_p\times \alpha_p$
(see \S \ref{sec:31}). For each member $(A_1,\lambda_1)$ in
$\Lambda_{2,p}^*$, the space of degree-$p$ isogenies
$\varphi:(A_1,\lambda_1)\to 
(A,\lambda)$ with $\varphi^*\lambda=\lambda_1$ over $k$ is a 
projective line $\bfP^1$ over $k$. Write $\bfP^1_{A_1}$ to indicate
the space of $p$-isogenies arising from $A_1$. This family is studied in
Moret-Bailly \cite{moret-bailly:p1}, and also in Katsura-Oort
\cite{katsura-oort:surface}. 
One defines an $\F_{p^2}$-structure on $\bfP^1$ using the
$W(\F_{p^2})$-structure of $M_1$ defined by $F^2=-p$, where $M_1$ is
the covariant \dieu module of $A_1$
and $F$ is the absolute Frobenius.
For any supersingular principally
polarized abelian surface $(A,\lambda)$ there exist an $(A_1,\lambda_1)$ in
$\Lambda_{2,p}^*$ and a degree-$p$ isogeny $\varphi:(A_1,\lambda_1)\to
(A,\lambda)$ with $\varphi^*\lambda=\lambda_1$. 
The choice of $(A_1,\lambda_1)$ and $\varphi$ may not be unique. 
However, the degree $[\F_{p^2}(\xi):\F_{p^2}]$ of the point
$\xi\in \bfP^1_{A_1}(k)$ that corresponds to $\varphi$ is well-defined.

In this paper we prove
\begin{thm}\label{11} Let $x = (A,\lambda)$ be a supersingular principally
  polarized abelian surface over $k$. Suppose that $(A,\lambda)$ is
  represented by a pair $(\ul A_1, \xi)$, where $\ul A_1\in \Lambda^*_{2,p}$
  and $\xi\in \bfP^1_{A_1}(k)$. Then
  \begin{equation*}
  \Mass(\Lambda_x)=\frac{L_p}{5760},
  \end{equation*}
where
\[ L_p=
\begin{cases}
  (p-1)(p^2+1) & \text{if $\F_{p^2}(\xi)=\F_{p^2}$}, \\
  (p^2-1) (p^4-p^2)& \text{if $[\F_{p^2}(\xi):\F_{p^2}]=2$}, \\
  (p^2-1) |\PSL_2(\F_{p^2})| & \text{otherwise}.
\end{cases} \]
\end{thm}

Theorem~\ref{11} calculates the cardinality of $\ell$-adic Hecke
orbits $\calH_\ell(x)$, as one has (Corollary~\ref{23})
\[ |\calH_\ell(x)|=|\Sp_{2g}(\Z/N\Z)|\cdot\Mass(\Lambda_x). \] 
We mention that the function field analogue of Theorem~\ref{11} where
supersingular abelian surfaces are replaced by supersingular Drinfeld
modules is established in \cite{yu-yu:ssd}. 
  
This paper is organized as follows. In Section~\ref{sec:02} we
describe the relationship between supersingular $\ell$-adic Hecke
orbits and mass formulas. We develop the mass formula for the
orbits of certain superspecial abelian varieties. 
In Section~\ref{sec:03} we
compute the endomorphism ring of any supersingular abelian surface.
The proof of the main theorem is given in the last section.

\def\DS{{\rm DS}}

\section{Hecke orbits and mass formulas}
\label{sec:02}
\def\qisom{{\rm Q\text{-}isom}}
\newcommand{\double}[2]{#1(\Q)\backslash #1(\A_f)/#2}
\newcommand{\doublecoset}[1]{\double{#1}{#1(\hat \Z)}} 
Let $g, p, N, \ell, \calA_{g,1,N},k$ be as in the previous
section. We work with a slightly bigger moduli space in which
the objects are not necessarily equipped with principal
polarizations. 
It is indeed more convenient to work in this setting. 
Let $\calA_{g,p^*,N}=\cup_{m\ge 1} \calA_{g,p^m,N}$ be the moduli
space over $\Fpbar$ of $g$-dimensional abelian varieties
together with a  $p$-power degree polarization and a symplectic
level-$N$ structure with respect to $\zeta_N$. Write $\calA_{g,p^*}$
for the moduli stack over $\Fpbar$ that parametrizes $g$-dimensional
$p$-power degree polarized abelian varieties. 
For any point  $x=\ul A_0=(A_0,\lambda_0,\eta_0)$ in
$\calA_{g,p^*,N}(k)$, the $\ell$-adic Hecke orbit $\calH_\ell(x)$ is
defined to be the countable subset of $\calA_{g,p^*,N}(k)$ that consists of
points $\ul A$ such that there is an $\ell$-quasi-isogeny from $A$ to
$ A_0$ that preserves the polarizations. 
An $\ell$-quasi-isogeny from $A$ to
$ A_0$ is an element $\varphi\in \Hom(A,A_0)\otimes \Q$ such that
$\ell^m\varphi$, for some integer $m\ge0$, is an isogeny of
$\ell$-power degree. 

\subsection{Group theoretical interpretation}
Assume that $x$ is supersingular. Let $G_{x}$ be the
automorphism group scheme over $\Z$ associated to $\ul A_0$; for any
commutative ring $R$, the group of its $R$-valued points is defined by
\begin{equation*}
  G_{x}(R)=\{h\in (\End_k(A_{0})\otimes R)^\times \, | \, h' h=1\},
\end{equation*}
where $h\mapsto h'$ is the Rosati involution induced by
$\lambda_0$. 
%
%
Let $\Lambda_{x,N}\subset \calA_{g,p^*,N}(k)$ be the subset
consisting of objects $(A,\lambda,\eta)$ such that there is an
isomorphism $\epsilon_p: (A,\lambda)[p^\infty]\simeq
(A_0,\lambda_0)[p^\infty]$ of quasi-polarized $p$-divisible groups.  
Since $\ell$-quasi-isogenies do not change the associated
$p$-divisible group structure, we have the inclusion 
$\calH_\ell(x)\subset\Lambda_{x,N}$.
\begin{prop}\label{21}
Notations and assumptions as above.

(1) There is a natural isomorphism $\Lambda_{x,N}\simeq
    G_x(\Q)\backslash G_x(\A_f)/K_N$ of pointed sets, 
    where $K_N$ is the stabilizer of 
    $\eta_0$ in $G_x(\hat \Z)$. 

(2) One has $\calH_\ell(x)=\Lambda_{x,N}$.   
\end{prop}
\begin{proof}
  (1) This is a special case of \cite[Theorem 2.1 and Proposition
      2.2]{yu:mass_hb}. We sketch the proof for the reader's
      convenience. Let $\ul A$ be an element in $\Lambda_{x,N}$. As $A$ is
      supersingular, there is a quasi-isogeny $\varphi:A_0\to A$ such
      that $\varphi^*\lambda=\lambda_0$. For each prime $q$ (including
      $p$ and $\ell$), choose an isomorphism $\epsilon_q: \ul
      A_0[q^\infty]\simeq \ul A[q^\infty]$ of $q$-divisible groups 
      compatible
      with polarizations and level structures. There is an element
      $\phi_q\in G_x(\Q_q)$ such that $\varphi \phi_q=\epsilon_q$ for
      all $q$. The map $\ul A\mapsto [(\phi_q)]$ gives a well-defined map
      from $\Lambda_{x,N}$ to $G_x(\Q)\backslash G_x(\A_f)/K_N$. 
      It is not hard
      to show that this is a bijection. 

  (2) The inclusion $\calH_\ell(x)\subset \Lambda_{x,N}$ under the
      isomorphism in (1) is given by
\[   [G_x(\Q)\cap G_x(\hat \Z^{(\ell)})] \backslash 
[G_x(\Q_\ell)\times G_x(\hat \Z^{(\ell)})]/K_N 
\subset G_x(\Q)\backslash G_x(\A_f)/K_N.  \]
Since the group $G_x$ is semi-simple and simply-connected, the strong
approximation 
shows that $G_x(\Q)\subset G_x(\A^{(\ell)}_f)$ is dense. The equality
then follows immediately. \qed
\end{proof}

\begin{cor}\label{22}
  Let $\ul A_i=(A_i,\lambda_i,\eta_i)$, $i=1,2$, be two supersingular
  points in $\calA_{g,p^*,N}(k)$. Suppose that there is an isomorphism
  of the associated quasi-polarized $p$-divisible groups. Then for any
  prime $\ell\nmid pN$ there is an $\ell$-quasi-isogeny
  $\varphi:A_1\to A_2$ which preserves the polarizations and level 
  structures.
\end{cor}
%
%
\begin{proof}
  This follows from the strong approximation property for $G_x$ that
  any element $\phi$ in the double space $G_x(\Q)\backslash
  G_x(\A_f)/K_N$ can be represented by an element in
  $G_x(\Q_\ell)\times K^{(\ell)}_N$, where $K^{(\ell)}_N\subset
  G_x(\hat \Z^{(\ell)})$ is the prime-to-$\ell$ component of $K_N$. \qed
\end{proof}

Recall that we denote by $\Lambda_x$ the set
of isomorphism classes of $g$-dimensional supersingular $p$-power
degree polarized abelian varieties $(A,\lambda)$ over $k$ such that
there is an isomorphism 
$(A,\lambda)[p^\infty]\simeq (A_0,\lambda_0)[p^\infty]$, and 
define the mass $\Mass(\Lambda_x)$ of $\Lambda_x$ as
\begin{equation*}
  \Mass(\Lambda_x):=\sum_{(A,\lambda)\in \Lambda_x}
  \frac{1}{|\Aut(A,\lambda)|}.
\end{equation*}
Similarly, we define
\begin{equation*}
  \Mass(\Lambda_{x,N}):=\sum_{(A,\lambda,\eta)\in \Lambda_{x,N}}
  \frac{1}{|\Aut(A,\lambda,\eta)|}.
\end{equation*}

\begin{cor}\label{23}
  One has $|\calH_\ell(x)|=|\Sp_{2g}(\Z/N\Z)|\cdot \Mass(\Lambda_x)$. 
\end{cor}
\begin{proof}
  This follows from 
  \begin{equation*}
  |\calH_\ell(x)|=|\Lambda_{x,N}|=\Mass(\Lambda_{x,N})=|G_x(\Z/N\Z)|\cdot
  \Mass(\Lambda_x)   
  \end{equation*}
and $|G_x(\Z/N\Z)|=|\Sp_{2g}(\Z/N\Z)|$.\qed 
\end{proof}

\subsection{Relative indices}
Write $G'$ for the automorphism group scheme associated to a
principally polarized superspecial point $x_0$. The group $G'_\Q$ is
unique up to isomorphism. This is an inner form of $\Sp_{2g}$ which is
``twisted at $p$ and $\infty$''
(cf.~\S \ref{sec:31} below).
For any supersingular point $x\in
\calA_{g,p*}(k)$, we can regard $G_x(\Z_p)$ as an open compact
subgroup of $G'(\Q_p)$ through a choice of a quasi-isogeny of
polarized abelian varieties between $x_0$ and
$x$. Another choice of quasi-isogeny
gives rise to a subgroup which differs from the previous one 
by the conjugation of an element in
$G'(\Q_p)$. For any two open compact subgroups $U_1, U_2$ of
$G'(\Q_p)$, we put
\begin{equation*}
  \mu(U_1/U_2):=[U_1:U_1\cap U_2][U_2:U_1\cap U_2]^{-1}.
\end{equation*}

\begin{prop}\label{24}
  Let $x_1, x_2$ be two supersingular points in
  $\calA_{g,p^*}(k)$. Then one has
  \begin{equation*}
    \Mass(\Lambda_{x_2})=\Mass(\Lambda_{x_1})\cdot\mu
  (G_{x_1}(\Z_p)/G_{x_2}(\Z_p)).
  \end{equation*}
\[  \]
\end{prop}
\begin{proof}
  See Theorem 2.7 of \cite{yu:mass_hb}. \qed
\end{proof}

\subsection{The superspecial case}
Let $\Lambda_g$ denote the set of isomorphism classes of
$g$-dimensional principally polarized superspecial abelian varieties
over $\Fpbar$. When $g=2D>0$ is even, we denote by $\Lambda^*_{g,p^D}$ 
the set of isomorphism classes of $g$-dimensional polarized
superspecial abelian varieties $(A,\lambda)$ of degree $p^{2D}$
over $\Fpbar$ satisfying $\ker \lambda=A[F]$, where $F:A\to
A^{(p)}$ is the relative Frobenius morphism on $A$.  
Write
\begin{equation*}
  M_g:=\sum_{(A,\lambda)\in\Lambda_g}
  \frac{1}{|\Aut(A,\lambda)|}, \quad
  M_g^*:=\sum_{(A,\lambda)\in\Lambda^*_{g,p^D}} 
  \frac{1}{|\Aut(A,\lambda)|}
\end{equation*}
for the mass attached to the finite sets $\Lambda_g$
  and $\Lambda^*_{g,p^D}$, respectively.

\begin{thm}\label{25}
Notations as above. 

(1) For any positive integer $g$, one has
\begin{equation*}
   M_g=
  \frac{(-1)^{g(g+1)/2}}{2^g} \left \{ \prod_{k=1}^g \zeta(1-2k)
  \right \}\cdot \prod_{k=1}^{g}\left\{(p^k+(-1)^k\right \},
  \end{equation*}
where $\zeta(s)$ is the Riemann zeta function. 
  
(2) For any positive even integer $g=2D$, one has 
\begin{equation*}
   M^*_g=
  \frac{(-1)^{g(g+1)/2}}{2^g} \left \{ \prod_{k=1}^g \zeta(1-2k)
  \right \}\cdot \prod_{k=1}^{D}(p^{4k-2}-1).
  \end{equation*}
\end{thm}
\begin{proof}
  (1) This is due to Ekedahl and  Hashimoto-Ibukiyama (see
      \cite[p.159]{ekedahl:ss} and \cite[Proposition 9]
      {hashimoto-ibukiyama:classnumber}, also
      cf. \cite[Section 3]{yu:ss_siegel}).

  (2) See Theorem 6.6 of \cite{yu:ss_siegel}. \qed
\end{proof}

\begin{cor}\label{26} One has
\begin{equation*}
  M_2=\frac{(p-1)(p^2+1)}{5760}, \quad \text{and}\quad 
  M_2^*=\frac{(p^2-1)}{5760}.
  \end{equation*}
\end{cor}

\begin{proof}
  This follows from Theorem~\ref{25} and the basic fact
  $\zeta(-1)=\frac{-1}{12}$ and $\zeta(-3)=\frac{1}{120}$. This is
  also obtained in Katsura-Oort 
  \cite[Theorem 5.1 and Theorem 5.2]{katsura-oort:surface} by a method
  different from above. \qed
\end{proof}

\begin{remark} \label{27} Proposition~\ref{21} is generalized to the moduli
  spaces of PEL-type in \cite{yu:smf}, with modification due to the
  failure of the Hasse principle.
\end{remark}

\section{Endomorphism rings}
\label{sec:03}
\newcommand{\DM}{\mathcal{DM}}

In this section we treat the endomorphism rings of supersingular
abelian surfaces.

\subsection{Basic setting}
\label{sec:31}
For any abelian variety $A$ over $k$, the {\it
  $a$-number} $a(A)$ of $A$ is defined by
\begin{equation*}
  a(A):=\dim_k \Hom(\alpha_p, A).
\end{equation*}
Here $\alpha_p$ is the kernel of the Frobenius morphism $F:\Ga \to
\Ga$ on the additive group.
Denote by $\DM$ the category of \dieu modules over $k$.
If $M$ is the (covariant) \dieu module of
$A$, then 
\[ a(A)=a(M):=\dim_k M/(F,V)M. \]
 
Let $B_{p,\infty}$ denote the quaternion algebra over $\Q$ which is ramified
exactly at $\{p,\infty\}$. Let $D$ be the division quaternion algebra
over $\Q_p$ and $O_D$ be the maximal order. Let $W=W(k)$ be the ring
of Witt vectors over $k$,
$B(k):=\Frac (W(k))$ the fraction
field, and $\sigma$ the Frobenius
map on $W(k)$. We also write
$\Q_{p^2}$ and $\Z_{p^2}$ for $B(\F_{p^2})$ and $W(\F_{p^2})$,
respectively.  

Let $A$ be an abelian variety (over any field). 
The endomorphism ring $\End(A)$ is an
order of the semi-simple algebra $\End(A)\otimes \Q$. Determining
$\End(A)$ is equivalent to determining the semi-simple algebra
$\End(A)\otimes \Q$ and all local orders $\End(A)\otimes
\Z_\ell$. Suppose that $A$ is a supersingular abelian variety over
$k$. We know that 
\begin{itemize}
\item $\End(A)\otimes \Q = M_g(B_{p,\infty})$, and 
\item $\End(A)\otimes \Z_\ell=M_{2g}(\Z_\ell)$ for all primes $\ell\neq p$.
\end{itemize}
Therefore, it is sufficient to determine the local
endomorphism ring $\End(A)\otimes \Z_p=\End_{\DM}(M)$, which is an
order of the simple algebra $M_g(D)$.

\subsection{The surface case}
\label{sec:32}
Let $A$ be a supersingular abelian surface
over $k$. There is a superspecial abelian surface $A_1$ and an isogeny
$\varphi:A_1\to A$ of degree $p$. Let $M_1$ and $M$ be the covariant \dieu
modules of $A_1$ and $A$, respectively. One regards $M_1$ as a
submodule of $M$ through the injective map $\varphi_*$. 
Let $N$ be the \dieu submodule in
$M_1\otimes \Q_p$ such that $VN=M_1$. If $a(M)=1$, then $M_1=(F,V)M$
and hence it is determined by $M$. If $a(M)=2$, or equivalently $M$ is
superspecial, then there are $p^2+1$ superspecial submodules 
$M_1\subset M$ such that $\dim_k M/M_1=1$.
    
Now we fix a rank 4 superspecial \dieu module $N$ (and hence fix $M_1$) 
and consider the
space $\calX$ of \dieu submodules $M$ with $M_1\subset M\subset N$ and
$\dim_k N/M=1$. It is clear that $\calX$ is isomorphic to the
projective line  $\bfP^1$ over $k$. Let $\wt N\subset N$ be the
$W(\F_{p^2})$-submodule defined by $F^2=-p$. This gives an
$\F_{p^2}$-structure on $\bfP^1$. It is easy to show the following

\begin{lemma}\label{31}
  Let $\xi\in \bfP^1(k)$ be the point corresponding to a \dieu module
  $M$ in $\calX$. Then $M$ is superspecial if and only if $\xi\in
  \bfP^1(\F_{p^2})$. 
\end{lemma}

\def\Span{{\rm Span}}
Choose a $W$-basis $e_1, e_2, e_3, e_4$ for $N$ such that 
\begin{equation*}
  Fe_1=e_2, \quad Fe_2=-pe_1,\quad Fe_3=e_4, \quad Fe_4=-pe_3. 
\end{equation*}
Note that this is a $W(\F_{p^2})$-basis for $\wt N$.
Write $\xi=[a:b]\in \bfP^1(k)$. The corresponding \dieu module
$M$ is given by
\begin{equation*}
  M=\Span <pe_1, pe_3, e_2, e_4, v>,
\end{equation*}
where $v=a'e_1+b'e_3$ and $a',b'\in W$ are any liftings of $a, b$
respectively. \\

{\bf Case (i):} $\xi\in \bfP^1(\F_{p^2})$. In this case $M$ is
superspecial. We have $\End_{\DM}(M)=M_2(O_D)$.\\

Assume that $\xi\not\in \bfP^1(\F_{p^2})$. In this case $a(M)=1$. If
$\phi\in \End_{\DM}(M)$, then $\phi\in \End_{\DM}(N)$. Therefore,
\begin{equation*}
  \End_{\DM}(M)=\{\, \phi\in \End_{\DM}(N)\, ;\ \phi(M)\subset M\, \}.   
\end{equation*}
We have $\End_{\DM}(N)=\End_{\DM}(\wt N)=M_2(O_D)$. The induced map 
\begin{equation}
  \label{eq:35}
  \pi:\End_{\DM}(\wt N)\to \End_{\DM}(\wt N/V\wt N) 
\end{equation}
is surjective. Put 
\begin{equation*}
  V_0:=\wt N/V\wt N=\F_{p^2}e_1\oplus \F_{p^2}e_3\quad
  \text{and}\quad B_0:=\End_{\F_{p^2}}(V_0).
\end{equation*}
We have 
\begin{equation*}
  \End_{\DM}(\wt N/V\wt N)=\End_{\F_{p^2}}(V_0)=M_2(\F_{p^2}).
\end{equation*}
Put 
\begin{equation*}
\quad B'_0:=\{T\in B_0\, ;\, T(v)\in k\cdot v\,\},
\end{equation*}
where $v=ae_1+be_3\in V_0\otimes_{\F_{p^2}} k$. Therefore, 
$\End_{\DM}(M)=\pi^{-1}(B'_0)$. Since $\xi\not\in \bfP^1(\F_{p^2})$,
$a\neq 0$. We write $\xi=[1:b]$, $v=e_1+b e_3$, and we have
$\F_{p^2}(\xi)=\F_{p^2}(b)$. Write $T=
\begin{pmatrix}
  a_{11} & a_{12} \\ a_{21} & a_{22}
\end{pmatrix} \in B_0$, where $a_{ij}\in \F_{p^2}$. 
From $T(v)\in k v$, we get the condition
\begin{equation}
  \label{eq:39}
  a_{12}b^2+(a_{11}-a_{22})b-a_{21}=0.
\end{equation} 

{\bf Case (ii): $\F_{p^2}(\xi)/\F_{p^2}$ is quadratic.} Write $\xi=[1:b]$.
Suppose $b$ satisfies $b^2=\alpha b+\beta$, where $\alpha,\beta\in
\F_{p^2}$. Plugging this in (\ref{eq:39}), we get 
\begin{equation*}
a_{11}-a_{12}+a_{12}
\alpha=0 \quad \text{and}\quad  a_{12}\beta=a_{21}. 
\end{equation*}
This shows
\begin{equation}\label{eq:312}
  B_0'=\left \{ t_1 I+t_2
  \begin{pmatrix}
    0 & 1 \\ \beta & \alpha
  \end{pmatrix}\,;\, t_1, t_2\in \F_{p^2}\,\right \}\simeq \F_{p^2}(\xi),
\end{equation}
where $X^2-\alpha X-\beta$ is the minimal polynomial of $b$.

{\bf Case (iii): $\xi\not\in \bfP^1(\F_{p^2})$ and
  $\F_{p^2}(\xi)/\F_{p^2}$ is not quadratic.} In this 
case $a_{12}=a_{21}=0$ and $a_{11}=a_{22}$. We have
\begin{equation*}
  B_0'=\left \{
  \begin{pmatrix}
    a & 0 \\ 0 & a
  \end{pmatrix}\,;\, a\in \F_{p^2}\,\right \}.
\end{equation*}

We conclude

\begin{prop}\label{32}
  Let $A$ be a supersingular surface over $k$ and $M$ be the
  associated covariant \dieu module. Suppose that $A$ is
  represented by a pair $(A_1,\xi)$, where $A_1$ is a superspecial
  abelian surface and $\xi\in \bfP^1_{A_1}(k)$.
  Let $\pi:M_2(O_D)\to M_2(\F_{p^2})$ be the
  natural projection.

(1) If $\F_{p^2}(\xi)=\F_{p^2}$, then $\End_{\DM}(M) = M_2(O_D)$.

(2) If $[\F_{p^2}(\xi):\F_{p^2}]=2$, then 
\begin{equation*}
  \End_{\DM}(M)\simeq \{\phi\in M_2(O_D)\, ; \pi(\phi)\in B'_0\, \},
\end{equation*}
where $B_0'\subset M_2(\F_{p^2})$ is a subalgebra
isomorphic to $\F_{p^2}(\xi)$. 

(3) If it is not in the case (1) or (2), then 
\begin{equation*}
  \End_{\DM}(M)\simeq \left \{\phi\in M_2(O_D)\, ; \pi(\phi)=  \begin{pmatrix}
    a & 0 \\ 0 & a
  \end{pmatrix},\, a\in \F_{p^2} \, \right \}.
\end{equation*}
\end{prop}

\section{Proof of Theorem~\ref{11}}
\label{sec:04}

\subsection{The automorphism groups}
\label{sec:41}

Let $x=(A,\lambda)$ be a supersingular principally polarized abelian
surfaces over $k$. Let $x_1=(A_1,\lambda_1)$ be an element in
$\Lambda^*_{2,p}$ such that there is a degree-$p$ isogeny
$\varphi:(A_1,\lambda_1)\to (A,\lambda)$ of polarized abelian
varieties. Write $\xi=[a:b]\in \bfP^1(k)$ the point corresponding to the
isogeny $\varphi$. We choose an $\F_{p^2}$-structure on $\bfP^1$ as
in \S \ref{sec:32}.
Let $(M_1, \<\, , \>)\subset (M, \<\, , \>)$ be the covariant \dieu
modules associated to $\varphi:(A_1,\lambda_1)\to (A,\lambda)$. Let
$N$ be the submodule in $M_1\otimes \Q_p$ such that $VN=M_1$, and put
$\<\, , \>_N=p \<\, , \>$. One has an isomorphism $(N,\<\, , \>_N)\simeq
(M_1,\<\, , \>)$ of quasi-polarized \dieu modules. Put 
\begin{equation*}
  \begin{array}[c]{c}
  U_{x}:=G_x(\Z_p)=\Aut_{\DM}(M,\<\, , \>),\\
  U_{x_1}:=G_{x_1}(\Z_p)=\Aut_{\DM}(M_1,\<\, , \>)=\Aut_{\DM}(N,\<\, , \>_N). 
  \end{array}
\end{equation*}

Choose a $W$-basis $e_1, e_2, e_3, e_4$ for $N$ such that 
\begin{equation*}
  \begin{array}[c]{c}
  Fe_1=e_2, \quad Fe_2=-pe_1,\quad Fe_3=e_4, \quad Fe_4=-pe_3, \\
  \<e_1, e_3\>_N=-\<e_3, e_1\>_N=1, \quad \<e_2, e_4\>_N=-\<e_4, e_2\>_N=p, 
  \end{array}
\end{equation*}
and $\<e_i, e_j\>=0$ for all remaining $i,j$. The \dieu module $M$ is
given by
\begin{equation*}
  M=\Span <pe_1, pe_3, e_2, e_4, v>,
\end{equation*}
where $v=a'e_1+b'e_3$ and $a',b'\in W$ are any liftings of $a, b$
respectively. \\

{\bf Case (i): $\xi\in \bfP^1(\F_{p^2}$)}. In this case $A$ is
superspecial. One has $\Lambda_x=\Lambda_2$ and, by Corollary~\ref{26}, 
\begin{equation*}
  \Mass(\Lambda_x)=\frac{(p-1)(p^2+1)}{5760}. 
\end{equation*}

In the remaining of this section, we treat the case $\xi\not\in
\bfP^1(\F_{p^2})$. One has
\begin{equation*}
  U_x=\{\phi\in U_{x_1}\, ;\, \phi(M)=M\, \},
\end{equation*}
and, by Proposition~\ref{24} and Corollary~\ref{26},
\begin{equation}
  \label{eq:46}
 \Mass(\Lambda_{x})=\Mass(\Lambda_{x_1})\cdot\mu
  (U_{x_1}/U_x)=\frac{p^2-1}{5760} [U_{x_1}:U_x].   
\end{equation}

Recall that $V_0=\wt N/V\wt N$, which is 
equipped with the non-degenerate alternating pairing 
$\<\, ,\>:V_0\times V_0\to \F_{p^2}$ induced
from $\<\, ,\>_N$. The map (\ref{eq:35}) induces a group homomorphism
\begin{equation*}
  \pi: U_{x_1}\to \Aut(V_0,\<\, ,\>)=\SL_2(\F_{p^2}). 
\end{equation*}
\begin{prop}\label{41}
  The map $\pi$ above is surjective. 
\end{prop}

The proof is given in Subsection~\ref{sec:42}. 


\begin{lemma}\label{42}
  One has $\ker \pi\subset U_x$.
\end{lemma}
\begin{proof}
  Let $\phi\in \ker \pi$. Write $\phi(e_1)=e_1+f_1,
  \phi(e_3)=e_3+f_3$, where 
  $f_1,f_3\in VN$. Since $M$ is generated by $VN$ and $v$, it suffices
  to check $\phi(v)=v+a'f_1+b'f_3\in M$; this is clear. \qed
\end{proof}

{\bf Case (ii): $[\F_{p^2}(\xi):\F_{p^2}]= 2$.} By Proposition~\ref{32} and
  Lemma~\ref{42}, we have $\pi: U_{x_1}/U_x\simeq
  \SL_2(\F_{p^2})/\F_{p^2}(\xi)^\times_1.$ This shows 
\[ [U_{x_1}: U_x]=(p^4-p^2). \]

{\bf Case (iii): $[\F_{p^2}(\xi):\F_{p^2}]\ge 3$.} By Proposition~\ref{32} and
  Lemma~\ref{42}, we have $\pi: U_{x_1}/U_x\simeq
  \SL_2(\F_{p^2})/\{\pm 1\}$. This shows 
\[ [U_{x_1}: U_x]=|\PSL_2(\F_{p^2})|. \]

From Cases (i)-(iii) above and equation (\ref{eq:46}), Theorem~\ref{11}
is proved. 

\subsection{Proof of Proposition~\ref{41}}
\label{sec:42}

Write 
\begin{equation*}
  O_D=W(\F_{p^2})[\Pi], \quad \Pi^2=-p,\quad \Pi a=a^\sigma \Pi,\
  \forall\,  a\in 
  W(\F_{p^2}).
\end{equation*}
The canonical involution is given by $(a+b\Pi)^*=a^\sigma-b\Pi$. With
the basis $1, \Pi$, we have the embedding 
\begin{equation*}
  O_D\subset M_2(W(\F_{p^2})), \quad a+b\Pi=
  \begin{pmatrix}
    a & -p b^\sigma \\ b & a^\sigma
  \end{pmatrix}.
\end{equation*}
Note that this embedding is compatible with the canonical involutions. 
With respect to the basis $e_1,e_2,e_3,e_4$, an element $\phi\in
\End_{\DM}(N)$ can be written as
\begin{equation*}
  T=(T_{ij})\in M_2(O_D)\subset M_4(W(\F_{p^2})), \quad
  T_{ij}=a_{ij}+b_{ij}\Pi=\begin{pmatrix}
    a_{ij} & -p b_{ij}^\sigma \\ b_{ij} & a_{ij}^\sigma
  \end{pmatrix}. 
\end{equation*}

Since $\phi$ preserves the pairing $\<\, ,\>_N$, we get the condition
in $M_4(\Q_{p^2})$:
\begin{equation}
  \label{eq:411}
  T^t 
  \begin{pmatrix}
    & J \\ -J & 
  \end{pmatrix} T=\begin{pmatrix}
    & J \\ -J & 
  \end{pmatrix}, \quad J=
  \begin{pmatrix}
    1 & \\
    & p 
  \end{pmatrix}. 
\end{equation}
Note that 
\begin{equation*}
  w_0 T_{ji}^* w_0^{-1}=T_{ji}^t, \quad w_0=
  \begin{pmatrix}
     & -1 \\1 & 
  \end{pmatrix}\in M_2(\Z_{p^2}). 
\end{equation*}
The condition (\ref{eq:411}) becomes
\begin{equation}
  \label{eq:413}
  \begin{pmatrix}
    w_0 & \\ & w_0
  \end{pmatrix}
  T^*  \begin{pmatrix}
    w_0^{-1} & \\ & w_0^{-1}
  \end{pmatrix}
  \begin{pmatrix}
    & J \\ -J & 
  \end{pmatrix} T=\begin{pmatrix}
    & J \\ -J & 
  \end{pmatrix}. 
\end{equation}
Since
\begin{equation*}
  \begin{pmatrix}
    w_0^{-1} & \\ & w_0^{-1}
  \end{pmatrix}
  \begin{pmatrix}
    & J \\ -J & 
  \end{pmatrix} =\begin{pmatrix}
    & -\Pi \\ \Pi & 
  \end{pmatrix} =\Pi \begin{pmatrix}
    & -1 \\ 1 & 
  \end{pmatrix}\in M_2(O_D),  
\end{equation*} 
we have 

\begin{lemma}\label{43}
  The group $U_{x_1}$ is the group of $O_D$-linear automorphisms on
  the standard $O_D$-lattice $O_D\oplus O_D$ which preserve that
  quaternion hermitian form $
  \begin{pmatrix}
    0 & -\Pi \\ \Pi & 0
  \end{pmatrix}$.
\end{lemma}
We also write (\ref{eq:413}) as 
\begin{equation}
  \label{eq:415}
  \Pi^{-1} T^* \Pi w T= w, \quad w=
  \begin{pmatrix}
    & -1 \\ 1 &  
  \end{pmatrix}\in M_2(O_D). 
\end{equation}

Notation.
For an element $T \in M_m(D)$ and
$n\in \Z$, write $T^{(n)}=\Pi^n T \Pi^{-n}$.
In particular,
if $T=(T_{ij})\in M_m (\Q_{p^2})\subset M_m(D)$, then
$T^{(n)}=(T_{ij}^{\sigma^n})$.
If $T \in M_m(O_D)$,
denote by $\ol{T} \in M_m(\F_{p^2})$
the reduction mod $\Pi$.

Suppose we are given $\bar \phi\in \SL_{2}(\F_{p^2})$, we must find
an element $T\in M_2(O_D)$ satisfying (\ref{eq:415}). We show that 
there is a sequence of  elements $T_n\in M_2(O_D)$ 
for $n\ge 0$ satisfying the conditions  
\begin{equation}
  \label{eq:416}
  (T_n^*)^{(1)} w T_n\equiv w \pmod{\Pi^{n+1}},\quad
  T_{n+1}\equiv T_n \pmod{\Pi^{n+1}}, \quad
  \text{and}\quad \ol T_0=\ol \phi.   
\end{equation}
Suppose there is already an element $T_n \in
M_2(O_D)$ for some $n\ge 0$ that satisfies
\begin{equation*}
  (T_n^*)^{(1)} w T_n\equiv w \pmod{\Pi^{n+1}}.  
\end{equation*}
Put $T_{n+1}:=T_n+B_n \Pi^{n+1}$, where $B_n\in M_2(O_D)$, and put $X_n:=
(T_n^*)^{(1)} w T_n$. Suppose $X_n\equiv w+C_n \Pi^{n+1}
\pmod{\Pi^{n+2}}$. One computes that
\begin{equation*}
  \begin{split}
    X_{n+1}& \equiv T^{*(1)}_n w T_n+T^{*(1)}_n w B_n \Pi^{n+1}+(\Pi^{n+1})^*
    B_n^{*(1)} w T_n \pmod{\Pi^{n+2}}\\
    & \equiv w+C_n \Pi^{n+1}+ T^{*(1)}_n w B_n \Pi^{n+1}+(-1)^{n+1}B_n^{*(n)}
    w T_n^{(n+1)} \Pi^{n+1} \pmod{\Pi^{n+2}}. 
  \end{split}
\end{equation*}
Therefore, we require an element $B_n\in M_2(O_D)$ satisfying
\begin{equation*}
  \ol C_n+\ol T^{t}_n w \ol B_n +(-1)^{n+1}\ol B_n^{t (n+1)}
    w \ol T_n^{(n+1)}=0. 
\end{equation*}
Put $\ol Y_n:=\ol T^{t}_n w \ol B_n$. As $\ol Y_n^t=- \ol B_n^t w \ol
T^{}_n$, we need to solve the equation
\begin{equation*}
  \ol C_n+\ol Y_n+(-1)^n \ol Y_n^{t(n+1)}=0, 
\end{equation*} 
or equivalently the equation
\begin{equation*}
  \begin{cases}
    \ \ol C_n+\ol Y_n+ \ol Y_n^{t(1)}=0, & \text{if $n$ is
    even}, \\
    \ \ol C_n+\ol Y_n- \ol Y_n^{t}=0, & \text{if $n$ is
    odd}. 
  \end{cases}
\end{equation*} 
It is easy to compute that $X_n^*=-X_n^{(1)}$. From this it follows
that 
\[ (-1)^{n+1} C_n^{*(n+1)} \Pi^{n+1}\equiv -C_n^{(1)} \Pi^{n+1}
\pmod{\Pi^{n+2}}, \] 
or simply $(-1)^n \ol C_n^{t(n)}=\ol C_n^{(1)}$. This gives the condition
\begin{equation*}
  \begin{cases}
   \  \ol C^t_n=\ol C^{(1)}_n, & \text{if $n$ is even,}\\ 
   \  -\ol C^t_n=\ol C^{}_n, & \text{if $n$ is odd.}
  \end{cases}
\end{equation*}

By the following lemma, we prove the existence of $\{T_n\}$ satisfying
(\ref{eq:416}). Therefore, Proposition~\ref{41} is proved. 

\begin{lemma}\label{44}
  Let $C$ be an element in the matrix algebra $M_m(\F_{p^2})$. 

(1) If $\,C^t=C^{(1)}$, then there is an element $Y\in M_m(\F_{p^2})$
    such that $C+Y+Y^{t(1)}=0$. 

(2) If $-C^t=C$, then there is an element $Y\in M_m(\F_{p^2})$
    such that $C+Y-Y^{t}=0$. 
\end{lemma}
\begin{proof}
  The proof is elementary and hence omitted. \qed
\end{proof}

\begin{remark}
Theorem~\ref{11} also provides another way to look at the
supersingular locus $S_2$ of the Siegel threefold. We used to divide
it into two parts: superspecial locus and non-superspecial
locus. Consider the mass function 
\[ M: S_2\to \Q, \quad x \mapsto \Mass (\Lambda_x). \]
Then the function $M$ divides the supersingular locus $S_2$ 
into 3 locally closed subsets that refine the previous one. 
More generally, we can consider the same function $M$ on the supersingular
locus $S_g$ of the Siegel modular variety of genus $g$. The situation
definitely becomes much more complicated. However, it is worth knowing
whether the following question has the affirmative answer.

 {\bf (Question):} Is the map $M:S_g\to \Q$ a constructible function?  
\end{remark}

\npr{\it Acknowledgments.} The research of C.-F.~Yu was partially
supported by the grants NSC 97-2115-M-001-015-MY3 and AS-98-CDA-M01. 
The research of J.-D. Yu was partially supported by the grant NSC
97-2115-M-002-018.

\end{document}